\documentclass{article}

\usepackage[utf8]{inputenc}
\usepackage{dsfont}

\usepackage{enumerate}
\usepackage{amssymb}
\usepackage{amsmath}
\usepackage{mathtools}
\usepackage{amsthm}
\usepackage{graphicx}
\usepackage{wrapfig}
\usepackage{lipsum}
\usepackage{float}
\usepackage{subcaption}
\usepackage{todonotes}
\usepackage{soul}
\usepackage{authblk}
\usepackage{geometry}

\theoremstyle{theorem}
\newtheorem{theorem}{Theorem}[section]
\theoremstyle{definition}
\newtheorem{definition}[theorem]{Definition}
\theoremstyle{proposition}
\newtheorem{proposition}[theorem]{Proposition}
\theoremstyle{lemma}
\newtheorem{lemma}[theorem]{Lemma}
\theoremstyle{corollary}

\theoremstyle{remark}
\newtheorem{remark}[theorem]{Remark}

\newcommand{\Var}{{\mathrm{Var}}}
\newcommand{\invisible}[1]{}

\title{Two more ways of spelling Gini Coefficient with Applications}

\author[1]{Marta Milewska}
\author[2]{Remco van der Hofstad}
\author[1,2]{Bert Zwart}

\affil[1]{Centrum Wiskunde \& Informatica (CWI), P.O. Box 94079, 1090 GB Amsterdam, The Netherlands}
\affil[2]{Department of Mathematics and Computer Science, Eindhoven University of Technology, 5600 MB Eindhoven,
The Netherlands}

\begin{document}
\maketitle

\begin{abstract}
We describe two new ways of defining and interpreting the Gini coefficient.
We investigate the potential application of the Gini coefficient to epidemiological models focusing on the negative binomial distribution. To this end, we derive several asymptotic results for the Gini coefficient. Finally, applying both new representations, we compute the Gini coefficient for a few probability distributions justifying the statement that the Gini coefficient can be used also for models relying on distributions other than negative binomial.
\end{abstract}

\section{Introduction}
Variability is one of the most fundamental features in data analysis. Statistics provides us with many tools suitable for measuring variability, the most popular of which is the variance. However, it can be argued that in some cases, for example when dealing with a distribution that significantly deviates from the Gaussian curve, variance  misses to convey some important information. Therefore, different metric have been proposed, which also provide quantitative information if the variance is infinite. 

A particular metric known as the Gini mean difference (abbreviated to GMD), first proposed in 1912 by the Italian statistician, Corrado Gini. Since then GMD and the parameters that can be derived from it (such as the Gini coefficient) have received quite some interest. However, they are mainly applied to quantify and compare the income disparities among countries while it is being suggested that this application does not reflect the entire potential of the GMD.

Out of the Gini family, the most well-known index is the Gini coefficient, again mostly used in the area of income inequalities. It was developed independently of the GMD and defined with respect to the Lorenz curve (explained in more detail further), which definition until today remains the most well-known formulation of the Gini coefficient. In 1914, Corrado Gini discovered and documented the relationship between the Gini coefficient and the Gini mean difference which turns out to be very simple: the Gini coefficient equals the GMD divided by twice the mean, assuming that the latter is positive. As a result, the Gini coefficient ranges from 0 to 1, where 0 corresponds to perfect equality and 1 to the situation where only one observation is positive.

So far at least 14 distinct alternative representations have been described, each of them leading to a different interpretation. The readers interested in investigating these various representations further should refer to \cite{Gini_methodology} or for a more concise guide to \cite{dozen_ways}.

Our motivation for considering the Gini coefficient arises from the field of epidemiology, where it has been realized that, apart from the mean reproduction number, variability can also play an important role. Viruses such as COVID-19 are considered to be highly overdispersed (see \cite{Taleb2} and \cite{evidence_fat_tailed}). So far, describing overdispersion was restricted to a particular parametric model (see \cite{Lloyd}). We want to make the case that the Gini coefficient provides a measure of overdispersion that is not limited to any particular probability distribution, and is consistent with the insights provided by existing parametric models. 

As computing the Gini coefficient with its classic definition might by arduous for some probability distributions, we first add two new representations: a probabilistic and an analytical one. We then apply the newly invented methodology to compute the Gini coefficient for a few classic probability distributions. We link our findings with the potential application to measuring overdispersion by giving more attention to the negative binomial distribution, which is often used in epidemiology to model the number of secondary cases per infected. In particular, we examine its asymptotic behavior.

This paper is organised as follows: in Section \ref{section_new_rep} we derive two new representations of the Gini coefficient: a probabilistic and an analytical one. In Section \ref{gini_various_distr} we apply the newly invented methodology to compute the Gini coefficient for a few classic probability distributions: the Poisson distribution in Section \ref{Poi_distr}, the exponential distribution in Section \ref{exp_distr}, the geometric distribution in Section \ref{geom_distr}, the Pareto distribution in Section \ref{Pareto_distr}, the uniform distribution in Section \ref{unif_distr} and the negative binomial distribution in Section \ref{negbinom_distr}. When possible, we refer to other examples of computing the Gini coefficient for the mentioned distributions. In Section \ref{applications} we focus on applications of the Gini coefficient for the negative binomial distribution, which is very often used for epidemiological modeling. We explain why the Gini coefficient should be considered as a measure of overdispersion in infectious diseases. In Theorems \ref{lim_k_small} and \ref{lim_k_large}, we prove the asymptotic behavior of the Gini coefficient of the negative binomial distribution for small and large parameters, respectively. Lastly, in Section \ref{conclusion} we close with a conclusion.

\section{New representations of the Gini coefficient} \label{section_new_rep}

The Gini coefficient is originally defined based on the Lorenz curve, which plots the proportion of the total income of the population ($y$-axis) that is cumulatively earned by the bottom $x \%$ of the population. Thus, the line at $45$ degrees represents perfect equality of incomes. The Gini coefficient can then be thought of as the ratio of the area that lies between the line of equality and the Lorenz curve over the total area under the line of equality. Before its link with the Gini mean difference was discovered, the coefficient was actually called a `concentration ratio'. However, the graphical representation is not always that useful. In order to derive alternative expressions of the Gini index, we first recall its well-known \cite{Gini_methodology, dozen_ways} definition in terms of the mean difference. 

\begin{definition}[Gini coefficient] \label{def_gini_wiki}
The stochastic representation of the Gini coefficient of a random variable $X$ with finite mean $\mathbf{E}[X]$ is
\begin{equation} \label{Gini_classic}
    G(X) = \frac{\mathbf{E}[{|X_1 - X_2|}]}{2\mathbf{E}[X]},
\end{equation}
where $X_1$ and $X_2$ are independent copies of $X$.
\end{definition}

\subsection{Representation in terms of a excess random variable.}

Formula (\ref{Gini_classic}) is the most popular representation of the Gini coefficient. However, it turns out that for many probability distributions it yields a rather unpleasant form. Hence, we propose a more instructive representation in terms of excess random variable:

\begin{definition}[excess random variable]
Let $X$ be a non-negative random variable with $0<\mathbf{E}[X]<\infty$. We define $X^*$, the excess version of $X$, as
\begin{align} \label{def_size_biased}
    \mathbf{P}(X^* \geq x) = \frac{1}{\mathbf{E}[X]} \int_x^{\infty} \mathbf{P}(X > s) \hspace{0.1cm} \text{d}s, \hspace{0.4cm} x \in \mathbb{R}_{\geq 0}
\end{align}
\end{definition}

\begin{theorem} [Gini coefficient in terms of excess variable] \label{thm_size_biased_representation}
For $X\geq0$ and $X^*$ independent,
\begin{align} \label{gini_formula}
    G(X) = \mathbf{P}(X^* \geq X),
\end{align}
where $X^*$ is a excess random variable defined in (\ref{def_size_biased}). 
\end{theorem}

\begin{proof}
By the definition of the Gini coefficient stated in (\ref{Gini_classic}),
\begin{align} \label{gini_rearrange}
    G(X) & = \frac{1}{2 \mathbf{E}[X]} \int_{0}^{\infty} \int_{0}^{\infty} |y-x| \hspace{0.1cm} \text{d}F_X(y) \hspace{0.1cm} \text{d}F_X(x) \nonumber \\
    & = \frac{1}{\mathbf{E}[X]} \int_{0}^{\infty}  \int_{x}^{\infty} (y-x) \hspace{0.1cm} \text{d}F_X(y) \hspace{0.1cm} \text{d}F_X(x).
\end{align}
We write $\int_{x}^{\infty} (y-x) \text{d}F_X(y) = \int_{x}^{\infty} \int_x^y \text{d}z \hspace{0.1cm} \text{d}F_X(y)$ and since the integrands are non-negative we can apply Fubini's theorem to change the order of integration in the latter and rearrange (\ref{gini_rearrange}) as
\begin{align*} 
    G(X) & = \frac{1}{\mathbf{E}[X]} \int_{0}^{\infty}  \int_x^{\infty} \int_z^{\infty} \text{d}F_X(y) \hspace{0.1cm} \text{d}z \hspace{0.1cm} \text{d}F_X(x) \\
    & = \int_{0}^{\infty}  \int_x^{\infty} \frac{\mathbf{P}(X > z)}{\mathbf{E}[X]} \hspace{0.1cm} \text{d}z \hspace{0.1cm} \text{d}F_X(x) = \mathbf{P}(X^* \geq X).
\end{align*}
\end{proof}

Denote $\mathbb{N}_0 = \mathbb{N} \cup \{0\}$. Theorem \ref{thm_size_biased_representation} is valid both for discrete and continuous random variables. However, for an $\mathbb{N}_0$-valued random variable, $X^*$ in (\ref{gini_formula}) can be replaced by a discrete version of excess random variable, $X^*_d$, with the probability distribution
\begin{align*}
    \mathbf{P}(X^*_d = x) = \frac{\mathbf{P}(X>x)}{\mathbf{E}[X]}, \hspace{0.5cm} x = 0, 1, 2, 3 \ldots
\end{align*}

\begin{proposition}
Take a random variable $X$ taking values in $\mathbb{N}_0$. Then, $X^* \stackrel{d}{=} X^*_d + Y$ with $Y \sim Unif[0,1]$, independent of $X_d^*$. Consequently, $G(X) = \mathbf{P}(X^*_d \geq X)$.
\end{proposition}

\begin{proof}
Observe first that $X^*_d + Y$ has a density function $f_{X^*_d+Y}$ such that
\begin{align*}
    f_{X^*_d+Y}(z) = \frac{\mathbf{P}(X>x)}{\mathbf{E}[X]} = p_x \hspace{0.5cm} \text{for} \hspace{0.2cm} z \in [x,x+1).
\end{align*}
We now differentiate (\ref{def_size_biased}) to obtain the density of $X^*$:
\begin{align*}
    f_{X^*}(z) & = \frac{\text{d}}{\text{d}z} \mathbb{P}(X^* \leq z) = \frac{\mathbf{P}(X>z)}{\mathbf{E}[X]} = p_x \hspace{0.3cm} \text{for} \hspace{0.1cm} z \in [x,x+1).
\end{align*}
Thus, $X^* \stackrel{d}{=} X^*_d + Y$. Next, observe that, for all $\mathbb{N}_0$-valued $X$,
\begin{align} \label{show_equality}
    \mathbf{P}(X^* \geq X) = \sum_{x=0}^{\infty} \mathbf{P}(X = x) \mathbf{P}(X^* \geq x) = \sum_{x=0}^{\infty} \mathbf{P}(X = x) \mathbf{P}(X^*_d + Y \geq x).
\end{align}
Since we only consider $x \in \mathbb{N}_0$, it is true that
\begin{align*}
    \mathbf{P}(X^*_d + Y \geq x) = \mathbf{P}(X^*_d \geq x) \cdot 1 + \mathbf{P}(X^*_d = x-1) \cdot \mathbb{P}(Y=1) = \mathbf{P}(X^*_d \geq x).
\end{align*}
Substituting this into (\ref{show_equality}) completes the proof of the second statement.
\end{proof}

Since using $X^*_d$ to compute the Gini index simplifies some calculations, throughout the paper we will use this definition when working with $\mathbb{N}_0$-valued random variables.

\subsection{Representation in terms of the Fourier Transform.}

By noticing that the Gini coefficient is a convolution at $0$ of the distribution of $X$ and the tail distribution of $-X^*$, one can arrive at a different representation using Fourier theory. Since Fourier theory is different in the discrete and continuous world, we make a distinction between the two cases.

\subsubsection{Discrete integer-valued case.}

The following theorem provides a representation of the Gini coefficient 
of a non-negative integer-valued random variable.

\begin{theorem} [Fourier representation of Gini coefficient in discrete case] \label{thm_Fourier_discrete}
Let $X$ be a nonnegative integer-valued random variable with probability mass function $p_X(j) = \mathbf{P}(X=j)$ and Fourier transform $\hat{p}_X(\theta) = \sum_{j=0}^{\infty} \mathbf{P}(X=j) e^{i \theta j}$. Then
\begin{align} \label{formula_one}
    G(X) = \frac{1}{4\pi \mathbf{E}[X]} \int_{0}^{2\pi} \frac{[1 - \hat{p}_X(\theta) - \mathbf{E}[X](e^{-i\theta} - 1)] \cdot \hat{p}_X(-\theta) }{1 - \cos{\theta}} \hspace{0.1cm} \text{d}\theta.
\end{align}
\end{theorem}

\begin{remark}
Since $ \hat{p}_X(\theta) \hat{p}_X(-\theta) = ||\hat{p}_X(\theta)||^2 $, (\ref{formula_one}) can be alternatively written as
\begin{align} \label{formula_one_alternative}
    G(X) = \frac{1}{4\pi \mathbf{E}[X]} \int_{0}^{2\pi} \frac{[1 - \mathbf{E}[X](e^{-i\theta} - 1)] \cdot \hat{p}_X(-\theta) - ||\hat{p}_X(\theta)||^2 }{1 - \cos{\theta}} \hspace{0.1cm} \text{d}\theta.
\end{align} 
\end{remark}

\begin{proof}
Write
\begin{align} \label{Gini_as_convolution}
    G(X) & = \sum_{j=0}^{\infty} \mathbf{P}(X=j) \mathbf{P}(X^*_d \geq j) = \sum_{j=0}^{\infty} \mathbf{P}(-X=-j) \mathbf{P}(X^*_d \geq j).
\end{align}
Now, if we define the functions $f_X$ and $g_X$ as
\begin{align*}
    f_X(j) = \mathbf{P}(-X = j) \hspace{0.3cm}  \text{and} \hspace{0.3cm} g_X(j) = \mathbf{P}(X^*_d \geq j),
\end{align*}
then we notice that (\ref{Gini_as_convolution}) is a convolution of these functions, say $(f_X \star g_X)(n)$, at $n=0$, so $G(X) = (f_X \star g_X)(0),$ which can be rearranged with the help of Fourier theory. Applying the inverse Fourier transform and the convolution theorem we obtain
\begin{align} \label{inverse_fourier}
    (f_X \star g_X)(n)
    & = \frac{1}{2\pi} \int_0^{2\pi} \hat{f_X}(\theta) \cdot \hat{g_X}(\theta) \cdot e^{-in \theta} \hspace{0.1cm} \text{d} \theta,
\end{align}
Substituting $n=0$ in (\ref{inverse_fourier}), we arrive at the new representation of $G(X)$ as
\begin{align} \label{gini_in_terms_of_Fourier}
    G(X) = \frac{1}{2\pi} \int_0^{2\pi} \hat{f}_X(\theta) \cdot \hat{g}_X(\theta) \hspace{0.1cm} \text{d}\theta.
\end{align}
It remains to compute the Fourier transforms $\hat{f}_X(\theta)$ and $\hat{g}_X(\theta)$. It holds that
\begin{align} \label{Fourier_f_theta_1}
    \hat{f}_X(\theta) = \sum_{j=-\infty}^{0} e^{i \theta j} \mathbf{P}(X=-j) = \hat{p}_X(-\theta).
\end{align}
Further,
\begin{align*} 
    \hat{g}_X(\theta) = \sum_{j=0}^{\infty} e^{i \theta j} \sum_{m=j}^{\infty} \mathbf{P}(X^*_d = m) = \sum_{m=0}^{\infty} \mathbf{P}(X^*_d = m) \frac{e^{i(m+1)\theta}-1}{e^{i \theta}-1},
\end{align*}
where the last equality follows from changing the order of summation and substituting $ \frac{e^{i(m+1)\theta}-1}{e^{i \theta}-1}=\sum_{j=0}^m  e^{i \theta j}$. Now, after extracting $\frac{1}{e^{i \theta}- 1}$ and splitting the sum we arrive at
\begin{align} \label{Fourier_g_theta_helpline3}
    \hat{g}_X(\theta) & = \frac{e^{i \theta}}{e^{i \theta}- 1} \sum_{m=0}^{\infty} \mathbf{P}(X^*_d = m) e^{i \theta m} - \frac{1}{e^{i \theta}- 1} \sum_{m=0}^{\infty} \mathbf{P}(X^*_d = m) \nonumber \\
    & = \frac{e^{i \theta}\hat{p}_{X^*_d}(\theta) - \hat{p}_{X^*_d}(0)}{e^{i \theta}- 1}=\frac{e^{i \theta}\hat{p}_{X^*_d}(\theta) - 1}{e^{i \theta}- 1}.
\end{align}
We now compute $\hat{p}_{X^*_d}(\theta)$ as
\begin{align*} 
    \hat{p}_{X^*_d}(\theta) & = \sum_{m=0}^{\infty} e^{im\theta} \frac{\mathbf{P}(X > m)}{\mathbf{E}[X]} = \frac{1}{\mathbf{E}[X]} \sum_{m=0}^{\infty} e^{im\theta} \sum_{l>m}^{\infty} \mathbf{P}(X = l) \\
    & = \frac{1}{\mathbf{E}[X]} \sum_{l=1}^{\infty} \mathbf{P}(X = l) \frac{e^{i \theta l}-1}{e^{i \theta}-1}=\frac{1}{\mathbf{E}[X]} \sum_{l=0}^{\infty} \mathbf{P}(X = l) \frac{e^{i \theta l}-1}{e^{i \theta}-1},
\end{align*}
where we again changed the order of summation and computed $\sum_{m=0}^{l-1} e^{im\theta}$ to obtain the last equality. After extracting $\frac{1}{e^{i \theta}-1}$ and splitting the sum we have
\begin{align} \label{Fourier*helpline2}
    \hat{p}_{X^*_d}(\theta) & = \frac{1}{\mathbf{E}[X] (e^{i \theta}-1)} \bigg( \sum_{l=0}^{\infty} \mathbf{P}(X = l)e^{i \theta l} - \sum_{l=0}^{\infty} \mathbf{P}(X = l) \bigg) \nonumber \\
    & = \frac{\hat{p}_X(\theta) - 1}{\mathbf{E}[X] (e^{i \theta}-1)}.
\end{align}
Substituting (\ref{Fourier*helpline2}) into (\ref{Fourier_g_theta_helpline3}) yields
\begin{align*}
    \hat{g}_X(\theta) & =  \frac{e^{i \theta}\hat{p}_{X^*_d}(\theta) - 1}{e^{i \theta}- 1} = \frac{e^{i \theta}\frac{\hat{p}_X(\theta) - 1}{\mathbf{E}[X] (e^{i \theta}-1)} - 1}{e^{i \theta}- 1} \\
    & = \frac{e^{i\theta}}{(e^{i\theta}-1)^2} \frac{\hat{p}_X(\theta) - 1}{\mathbf{E}[X]} - \frac{1}{e^{i\theta}-1}.
\end{align*}
Finally, we  rewrite
    \[
    \frac{e^{i\theta}}{(e^{i\theta}-1)^2}
    =\frac{1}{(e^{i\theta}-1)(1-e^{-i\theta})}
    =-\frac{1}{||e^{i\theta}-1||^2},
    \]
so that
\begin{align} \label{Fourier_g_theta_final_1}
    \hat{g}_X(\theta) = \frac{1-\hat{p}_X(\theta)-\mathbf{E}[X](e^{-i\theta}-1)}{||e^{i\theta}-1||^2 \mathbf{E}[X]}.
\end{align}  
We derive the final formula from Theorem \ref{thm_Fourier_discrete} by substituting (\ref{Fourier_f_theta_1}) and (\ref{Fourier_g_theta_final_1}) into (\ref{gini_in_terms_of_Fourier}) and noticing that $||e^{i\theta}-1||^2=\sin(\theta)^2+(\cos(\theta)-1)^2=2[1-\cos(\theta)].$
\end{proof}

\subsubsection{Continuous case.}
The following theorem provides an alternative representation of the Gini coefficient of a continuous non-negative random variable:

\begin{theorem} [Fourier representation of Gini coefficient in continuous case] \label{gini_continuous_fourier}
Let $X$ be a non-negative random variable with probability density function $p_X(x)$ and Fourier transform $\hat{p}_X(\theta) = \int_{0}^{\infty} p_X(x) e^{i\theta x} \hspace{0.1cm} \text{d}x$. Then
\begin{align} \label{gini_continuous_fourier_formula}
    G(X) = \frac{1}{2\pi \mathbf{E}[X]} \int_{\mathbf{R}} \frac{ \hat{p}_X(\theta)  (1 - \hat{p}_X(-\theta) - i \theta \mathbf{E}[X])}{\theta^2} \hspace{0.1cm} \mathrm{d}\theta.
\end{align}
\end{theorem}

\begin{proof}
The main idea in this proof is exactly the same as it was in the discrete case, though we need to define functions $f_X$ and $g_X$ in a slightly different way. Consider a convolution of $f_X(x), g_X(x)$ with
\begin{align*} 
f_X(x) = p_X(x) \hspace{0.2cm} \text{and} \hspace{0.2cm} g_X(x) = \mathbf{P}(-X^*\leq x) = \mathbf{P}(X^* \geq -x),
\end{align*}
where $g_X$ is defined for negative values of $x$. Applying analogous reasoning as in the proof of Theorem \ref{thm_Fourier_discrete}, we get
\begin{align} \label{convolution_formula}
    G(X) = \frac{1}{2\pi} \int_{\mathbf{R}} \hat{f}_X(\theta) \hspace{0.1cm} \hat{g}_X(\theta) \hspace{0.1cm} \text{d}\theta.
\end{align}
We again compute the corresponding Fourier transforms. We have
\begin{align*}
    \hat{g}_X(\theta) & = \int_{-\infty}^{0} \mathbf{P}(X^* \geq -x) e^{i\theta x} \hspace{0.1cm} \text{d}x = \int_{0}^{\infty} \mathbf{P}(X^* \geq y) e^{-i\theta y} \hspace{0.1cm} \text{d}y \nonumber \\
    & = \int_{0}^{\infty} e^{-i\theta y} \int_{y}^{\infty} p_{X^*}(z) \hspace{0.1cm} \text{d}z  \hspace{0.1cm} \text{d}y = \frac{1}{\mathbf{E}[X]} \int_{0}^{\infty} e^{-i\theta y} \int_{y}^{\infty} \mathbf{P}(X > z) \hspace{0.1cm} \text{d}z  \hspace{0.1cm} \text{d}y \nonumber \\
    & = \frac{i}{\theta \mathbf{E}[X]} \int_{0}^{\infty} \mathbf{P}(X > z)(e^{-i \theta z} - 1 ) \hspace{0.1cm} \text{d}z,
\end{align*}
where the last equality follows after changing the order of integration and computing $\int_{0}^{z} e^{-i\theta y} \hspace{0.1cm} \text{d}y = \frac{i}{\theta}(e^{-i \theta z} - 1 )$. We substitute $\mathbf{P}(X > z) = \int_{z}^{\infty} f_X(y) \hspace{0.1cm} \text{d}y$ and change the order of integration again to obtain
\begin{align*}
    \hat{g}_X(\theta) = \frac{i}{\theta \mathbf{E}[X]} \int_{0}^{\infty} f_X(y) \hspace{0.1cm} \int_{0}^{y} (e^{-i \theta z} - 1 ) \hspace{0.1cm} \text{d}z \hspace{0.1cm} \text{d}y.
\end{align*}
Note that $\int_{0}^{y} (e^{-i \theta z} - 1 ) \hspace{0.1cm} \text{d}z = \frac{i}{\theta}(e^{-i \theta y} - 1) - y $. Hence,
\begin{align*}
    \hat{g}_X(\theta) & = \frac{-1}{\theta^2 \mathbf{E}[X]} \int_{0}^{\infty} f_X(y) e^{-i \theta y} \hspace{0.1cm} \text{d}y + \frac{1}{\theta^2 \mathbf{E}[X]} \int_{0}^{\infty} f_X(y) \hspace{0.1cm} \text{d}y \nonumber \\
    &\qquad - \frac{i}{\theta \mathbf{E}[X]} \int_0^{\infty} y f_X(y) \hspace{0.1cm} \text{d}y, 
\end{align*}
which simplifies to
\begin{align} \label{g_final}
    \hat{g}_X(\theta) = \frac{1 - \hat{p}_X(-\theta) - i \theta \mathbf{E}[X]}{\theta^2 \mathbf{E}[X]}.
\end{align}
We obtain (\ref{gini_continuous_fourier_formula}) by substituting (\ref{g_final}) in (\ref{convolution_formula}) noting that $\hat{f}_X(\theta) = \hat{p}_X(\theta).$
\end{proof}

\section{Gini coefficient for various distributions} \label{gini_various_distr}

In this section, we apply our expressions to the computation of $G$ for various distributions. In some cases, the examples are new, and sometimes, our expressions provide easier derivations. 

\subsection{Poisson distribution} \label{Poi_distr}

Take $X \sim \text{Poisson}(\lambda)$. The characteristic function equals $\hat{p}_X(\theta)=\exp(\lambda (e^{i\theta}-1)).$ Consequently, $\hat{p}_X(-\theta) = \exp(\lambda (e^{-i\theta}-1))$ and we substitute these results into (\ref{formula_one}) to obtain
\begin{align}
    G(X) = \frac{1}{4 \lambda\pi} \int_{0}^{2\pi} \frac{[1 - \exp(\lambda (e^{i\theta}-1)) - \lambda (e^{-i\theta}-1)] \cdot \exp(\lambda (e^{-i\theta}-1)) }{1-\cos(\theta)} \hspace{0.1cm} \text{d}\theta. \nonumber
\end{align}
Alternatively, we could use (\ref{formula_one_alternative}). We compute
\begin{align*}
    ||\hat{p}_X(\theta) ||^2 & = ||\exp(\lambda (e^{i\theta}-1))||^2 \nonumber \\
    & = e^{-2\lambda} ||e^{\lambda(\cos{\theta}+i \sin{\theta})}||^2 = e^{-2\lambda(1-\cos{\theta})}.
\end{align*}
Substituting yields
\begin{align}
    G(X) = \frac{1}{4 \lambda\pi} \int_{0}^{2\pi} \frac{(1-\lambda(e^{-i\theta}-1)) e^{\lambda (e^{-i\theta}-1)} - e^{-2\lambda(1-\cos{\theta})}}{1-\cos(\theta)} \hspace{0.1cm} \text{d}\theta. \nonumber
\end{align}
The excess formula $G(X) = \mathbb{P}(X^*_d \geq X)$ does not seem to be convenient in this example.
An alternative representation of $G(X)$ in terms of Bessel functions can be found in \cite{Ramasubban}.

\subsection{Exponential distribution} \label{exp_distr}

\begin{proposition}
If $X \sim \text{Exp}(\lambda)$ then $G(X) = \frac{1}{2}$.
\end{proposition}

\begin{proof}
Note that $X^* \stackrel{d}{=} X$ and apply Theorem \ref{thm_size_biased_representation}.
\end{proof}

\subsection{Geometric distribution} \label{geom_distr}
Take $X \sim$ Geom$(p)$ with $\mathbf{P}(X=j) = (1-p)^j p$. It can be shown that $\mathbf{P}(X^*_d \geq j) = (1-p)^j$. We plug it in $G(X)$ to arrive at
\begin{align} \label{gini_geom}
    G(X) = p \sum_{j=0}^{\infty} (1-p)^j (1-p)^j = \frac{p}{p(2-p)} = \frac{1}{2-p}.
\end{align}
In case of geometric distribution Fourier representation is also easy to compute and yields the same result as (\ref{gini_geom}).

Alternatively, one can consider a shifted version of the Geometric distribution following \cite{gini_formulas_bibi}: 
take $Y \sim$ Geom'$(p)$ with $\mathbf{P}(Y=j) = p(1-p)^{j-1}$ supported on $\{1,2,3,...\}$. It follows that $\mathbf{P}(Y^*_d \geq j) = (1-p)^j$ and thus $G(Y) = \frac{p}{1-p} \sum_{j=1}^{\infty} (1-p)^{2j} = \frac{1-p}{2-p}.$ 

\subsection{Pareto distribution} \label{Pareto_distr}
Take $X \sim$ Pareto$(\alpha)$ with $f_X(x) = \frac{\alpha x_m^{\alpha}}{x^{\alpha+1}} = $, $\mathbf{E}X = \frac{\alpha x_m}{\alpha-1}$ and assume $\alpha > 1$. We first compute $\mathbf{P}(X^* \geq x)$ for $x \geq x_m$:
\begin{align*}
    \mathbf{P}(X^* \geq x) & = \frac{1}{\mathbf{E}[X]} \int_{y=x}^{+\infty} \mathbf{P}(X > y) \hspace{0.1cm}  \text{d}y = \frac{\alpha - 1}{\alpha x_m} \int_{y=x}^{+\infty} \bigg(\frac{x_m}{y} \bigg)^{\alpha} \hspace{0.1cm}  \text{d}y \nonumber \\
    & =  \frac{\alpha - 1}{\alpha} x_m^{\alpha-1} \bigg[ \frac{x^{1-\alpha}}{1-\alpha} \bigg]_{x}^{+\infty} = \frac{1}{\alpha} \bigg( \frac{x_m}{x} \bigg)^{\alpha-1}.
\end{align*}
Note that $\log\big(\frac{X}{x_m}\big) \sim Exp(\alpha)$, $\log\big(\frac{X^*}{x_m}\big)|X^*\geq x_m \sim Exp(\alpha-1)$. Since $\mathbf{P}(Y \geq Y') = \frac{\beta'}{\beta+\beta'}$, where $Y\sim Exp(\beta)$ and $Y'\sim Exp(\beta')$, we obtain
\begin{align*}
    \mathbf{P}(X^*\geq X) & = \mathbf{P}(X^*\geq X|X^*\geq x_m) \mathbf{P}(X^*\geq x_m) \\
    & = \mathbf{P}(\log\bigg(\frac{X^*}{x_m}\bigg)\geq \log\bigg(\frac{X}{x_m}\bigg)|X^*\geq x_m) \mathbf{P}(X^*\geq x_m) \\
    & = \frac{\alpha}{\alpha+\alpha-1} \cdot \frac{1}{\alpha} = \frac{1}{2\alpha-1}.
\end{align*}
This is not a new result, but its derivation is quite simple-- compare e.g. \cite[Example 1]{gini_formulas_bibi}.

\subsection{Continuous uniform distribution} \label{unif_distr}

We have not been able to find this explicit example in the literature.

\begin{proposition}
If $X \sim \mathcal{U}_{[a,b]}$, then $G(X) = \frac{b-a}{3(b+a)}$ = $\frac{2}{b-a} \cdot \frac{\Var X}{\mathbf{E}X}$.
\end{proposition}

\begin{proof}
\begin{align*}
    G(X) & = \int \mathbf{P}(X^* \geq x) f_X(x) \hspace{0.1cm} \text{d}x = \int_a^b \mathbf{P}(X^* \geq x) \frac{1}{b-a} \hspace{0.1cm} \text{d}x.
\end{align*}
Applying (\ref{def_size_biased}), substituting $\mathbf{E}X = \frac{a+b}{2}$ and changing the order of integration yields
\begin{align*}
    G(X) & = \frac{2}{(a+b)(b-a)} \int_a^{\infty} (y-a) \mathbf{P}(X > y) \hspace{0.1cm} \text{d}y \\
    & = \frac{2}{(a+b)(b-a)^2} \int_a^{b} (y-a) (b-y) \hspace{0.1cm} \text{d}y \\
    & = \frac{2}{(a+b)(b-a)^2}  \cdot \frac{1}{6} (b-a)^3 = \frac{b-a}{3(b+a)}.
\end{align*}
\end{proof}

\subsection{Negative binomial distribution} \label{negbinom_distr}

Let $X \sim NB(k,p)$ with probability mass function $\mathbf{P}(X = j) = \binom{k+j-1}{j}p^k (1-p)^j$ and $\mathbf{E}X = \frac{k(1-p)}{p}$. The classic definition of the Gini index (\ref{Gini_classic}) yields
\begin{align} \label{gini_neg_binom_ugly}
    G(X) = \frac{1}{p} \sum_{i=0}^{\infty} (-1)^i \binom{k+i}{i} \Bigg( \frac{1-p}{p^2} \Bigg)^i \frac{(2i)!}{i!(i+1)!}.
\end{align}
For a derivation see \cite{Ramasubban}, formula (2.11) for  the mean absolute difference of the negative binomial; then divide by twice the mean to obtain (\ref{gini_neg_binom_ugly}). The complexity of this expression was in fact our main motivation to develop new representations. Because of the alternating sign the formula is not very stable numerically and thus not handy to work with. The representation in terms of the excess random variable does not seem very helpful either. The computation leads to a rather non-intuitive formula containing three infinite sums. Instead, we apply Theorem \ref{thm_Fourier_discrete} to obtain the following representation:
\begin{theorem}
Take $X \sim NB(k,p)$ with probability mass function $\mathbf{P}(X = j) = \binom{k+j-1}{j}p^k (1-p)^j$ and $\mathbf{E}X = \frac{k(1-p)}{p}$, where $p$ is the probability of success. Then
\begin{align} \label{Gini_NB}
    & G(X) \nonumber = \\
    & \frac{p}{k(1-p)} \frac{1}{2\pi} \int_{0}^{2\pi} \frac{[1 - \frac{k(1-p)}{p} (e^{-i\theta}-1)] \cdot \bigg( \frac{pe^{i\theta}}{e^{i\theta}+p-1} \bigg)^k - \bigg( \frac{p^2}{p^2 + 2(p-1)(\cos{\theta}-1)} \bigg)^k }{ ||e^{i\theta}-1||^2} \hspace{0.1cm} \text{d} \theta.
\end{align}
\end{theorem}

\begin{remark}
Note that since $\binom{k+j-1}{j}$ can be interpreted as $\frac{\Gamma(k+j)}{\Gamma(j+1)\Gamma(k)}$, (\ref{Gini_NB}) also applies to non-integer $k$.
\end{remark}

\begin{proof}
We want to apply Theorem \ref{thm_Fourier_discrete} to the negative binomial variable $X$. Since
\begin{align} \label{term1}
    \hat{p}_{X}(\theta) = p^k (1 - (1-p)e^{i\theta})^{-k} = \frac{p^k}{(pe^{i\theta}-e^{i\theta}+1)^k},
\end{align}
then
\begin{align} \label{term2}
    \hat{p}_{X}(-\theta) = \frac{p^k}{(pe^{-i\theta}-e^{-i\theta}+1)^k} = \bigg( \frac{pe^{i\theta}}{e^{i\theta}+p-1} \bigg)^k.
\end{align}
Substituting (\ref{term1}) and (\ref{term2}) into (\ref{formula_one}) yields
\begin{align} \label{gini_neg_binom}
    & G(X)\nonumber \\
    & =\frac{p}{k(1-p)} \frac{1}{2\pi} \int_{0}^{2\pi} \frac{[1 - \frac{p^k}{(pe^{i\theta}-e^{i\theta}+1)^k} - \frac{k(1-p)}{p} (e^{-i\theta}-1)] \cdot \bigg( \frac{pe^{i\theta}}{e^{i\theta}+p-1} \bigg)^k }{||e^{i\theta}-1||^2} \hspace{0.1cm} \text{d}\theta.
\end{align}
We investigate the product of $\hat{p}_{X}(\theta)$ and $\hat{p}_{X}(-\theta)$:
\begin{align*}
     \bigg( \frac{p}{pe^{i\theta}-e^{i\theta}+1} \cdot \frac{pe^{i\theta}}{e^{i\theta}+p-1} \bigg)^k = \bigg( \frac{p^2e^{i\theta}}{(pe^{i\theta}-e^{i\theta}+1)(e^{i\theta}+p-1)} \bigg)^k.
\end{align*}
However, we have that
\begin{align*}
    (pe^{i\theta}- & e^{i\theta}+1)(e^{i\theta}+p-1) = pe^{2i\theta} + p^2e^{i\theta} - 2pe^{i\theta} - e^{2i\theta} + 2e^{i\theta} + p -1 \\
    & = e^{i\theta} \big(p^2-2p+2 + p(e^{i\theta} + e^{-i\theta}) - (e^{i\theta} + e^{-i\theta}) \big) \\
    & = e^{i\theta} \big( p^2 + 2(p-1)(\cos(\theta)-1) \big).
\end{align*}
Therefore
\begin{align*}
    \frac{p^2e^{i\theta}}{(pe^{i\theta}-e^{i\theta}+1)(e^{i\theta}+p-1)} = \frac{p^2}{p^2 + 2(p-1)(\cos{\theta}-1)}.
\end{align*}
After plugging in the above computation in (\ref{gini_neg_binom}) and rearranging we arrive at the expression (\ref{Gini_NB}).
\end{proof}

In the next section, we discuss the asymptotic behavior of the Gini coefficient of the negative binomial distribution as $k \to 0$ and $k \to \infty$.

\section{Applications to epidemics} \label{applications}
In epidemiology, the basic reproduction number $R_0$ denotes the expected number of new infections directly generated by one case. If $R_0>1$, the epidemic is growing, and the epidemic
is dying our when $R_0<1$. 
It has been determined that transmission patterns of SARS-COV-2 and many other viruses are highly overdispersed (see \cite{Taleb2} or \cite{evidence_fat_tailed}).  Looking at $R_0$ alone might therefore be misleading. In case of COVID-19 in particular, it is estimated that more than 80$\%$ of new infections resulted from the top 10$\%$ of most infectious individuals (see \cite{Endo}, \cite{overdispersion_Israel} or \cite{overdispersion_USA}). This means that the vast majority of infected individuals will not pass the infection on to anyone. $R_0$ should therefore be complemented  with another statistic that sheds light on possible overdispersion in the number of secondary cases. 

To this end, the authors in \cite{Lloyd} assume a parametric setting, in which  the number of secondary cases follows a negative binomial distribution with parameters $k$ and $R_0$, which results in its variance being equal to $R_0(1 + R_0/k)$. In such a setting, the smaller the value of $k$, the greater heterogeneity in the distribution of secondary infections. Despite coming from a very particular model, this interpretation became so popular that multiple papers adopted the symbol $k$ and referred to it as an `overdispersion parameter' (see for instance \cite{Endo}, \cite{Chen}, \cite{Eilersen} or \cite{Gini_Indonesia}). Even non-mathematical press followed this trend, calling $k$ the key to overcoming the pandemic of COVID-19 (see \cite{OverlookedVariable}).

In this section, we investigate the relation between $k$ and the Gini index. We show numerically that the Gini index, when specialized to the negative binomial distribution, depends monotonically on $k$, and complement this with rigorous asymptotic estimates for small and large $k$. This indicates that the Gini index provides insights consistent with the ones obtained so far, without the need to rely on a parametric model. Thus, the Gini index can serve as candidate to measure the variability in the infectiousness of a disease. 

This section is organised as follows: in Section \ref{Gini_asypmtotic} we first enable a continuous interpretation of parameter $k$ of the negative binomial distribution by expressing is as a composition of Poisson and Gamma processes. Next we investigate monotonicity of the Gini coefficient in $k$. To that end, we illustrate its behavior in plots and support them with statements about asymptotics of the Gini coefficient for $k \to 0$ and $k \to \infty$. In Section \ref{gini_data_example} we show for the real-world data set that the Gini coefficient provides insights consistent with using a parametric model based on $k$.

\subsection{Asymptotic behavior of the Gini coefficient} \label{Gini_asypmtotic}

In Section \ref{negbinom_distr} we showed that the Gini index for the negative binomial distribution cannot be easily obtained from (\ref{Gini_classic}), neither by its excess representation $\mathbb{P}(X^*_d \geq X).$ In this section, we investigate the Gini index for small and large values of $k$. 
For that purpose it is useful to first rewrite the negative binomial distribution in terms of the composition of Poisson and Gamma processes, which has the added benefit that the resulting interpretation is natural for non-integer values of $k$. 

Let $X$ follow a $Poisson(\nu)$ distribution, where the rate $\nu$ is also a random variable and it follows a $Gamma(k,\lambda)$ distribution. To maintain a clear connection to epidemiological notation, we denote $\mathbf{E}[\nu] = R_0$.

\begin{lemma} \label{neg_binom_vs_poisson_gamma}
For $X$ as above, it holds that $X \sim NB(k,p)$ with parameter $k$, corresponding to the number of `successes', and parameter $p=  \frac{\lambda}{\lambda +1}$ denoting the probability of success.
\end{lemma}
The proof of Lemma \ref{neg_binom_vs_poisson_gamma} is straightforward and is hence omitted. Note that $p= (1 + R_0/k)^{-1}$. In the following, we will write $\Gamma_k$ referring to a Gamma process determined by $\Gamma(k,\lambda)$ and $N(\Gamma_k)$ to denote a subordinated Poisson process with rate following this Gamma process.

With the new, continuous interpretation in mind we will have a closer look at the behaviour of the Gini coefficient as $k$ increases. For creating some visual intuition, Figure \ref{fig:gini_asympt_plots}  plots the Gini coefficient as a function of $k$ for different values of $p=\frac{\lambda}{\lambda+1}$.
\begin{figure}[h] 
    \centering
    \begin{subfigure}[b]{0.3\textwidth}   
        \centering 
        \includegraphics[width=\textwidth]{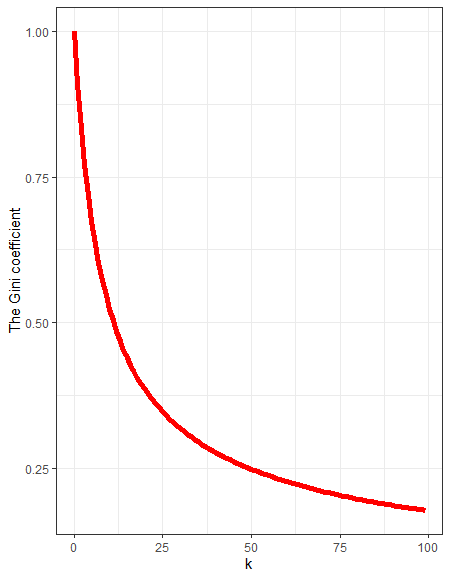}
            \caption[]%
            {{\small $p = 0.9$}}    
            %\label{fig:gini for k large 0.9}
    \end{subfigure}
    \hfill
    \begin{subfigure}[b]{0.3\textwidth}   
        \centering 
        \includegraphics[width=\textwidth]{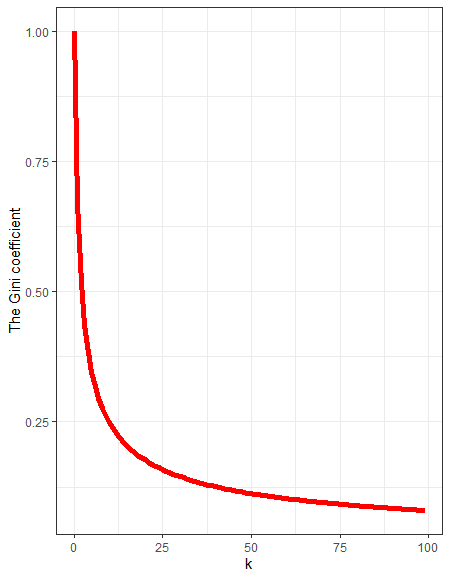}
            \caption[]%
            {{\small $p = 0.5$}}    
            %\label{fig:gini for k large 0.5}
    \end{subfigure}
    \hfill
    \begin{subfigure}[b]{0.3\textwidth}   
        \centering 
        \includegraphics[width=\textwidth]{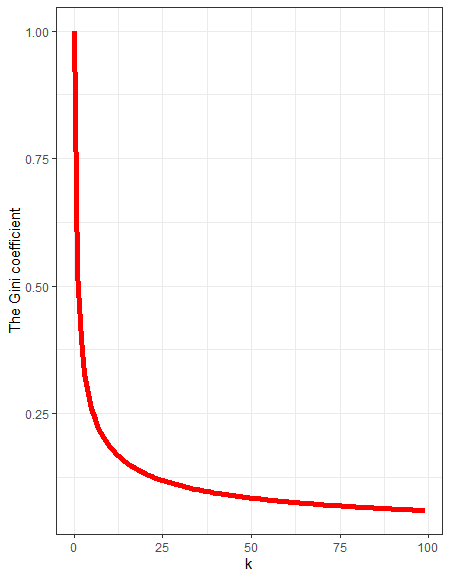}
            \caption[]%
            {{\small $p = 0.1$}}    
            %\label{fig:gini for k large 0.1}
    \end{subfigure}
    \caption{Plots of the Gini coefficient of negative binomial distribution} \label{fig:gini_asympt_plots}
\end{figure}
Figure \ref{fig:gini_asympt_plots} suggests that the Gini coefficient of the negative binomial is decreasing in $k$. Unfortunately, a rigorous proof of the statement remains an open problem. However, we succeeded in showing that the limiting behaviour of the Gini coefficient of the negative binomial distribution aligns with the desired statement, i.e., $G(NB(k,p))$ approaches $1$ as $k$ approaches $0$ and approaches $0$ as $k$ approaches infinity. We formulate these results in the following two theorems:

\begin{theorem} [Asymptotics of the Gini coefficient for $k$ small] \label{lim_k_small}
Let $X \sim NB(k,p)$. Then, for $k \to 0$,
\begin{align} \label{as_formula_k_small}
    G(NB(k,p)) = 1 + c \cdot k + o(k),
\end{align}
with $c=2\log(p)+\frac{2-p}{1-p}\log\big(2p(1-p)\big)=2\log \bigg(\frac{\lambda}{\lambda+1}\bigg) - (\lambda+2)\log\bigg(\frac{\lambda^2+2\lambda}{(\lambda+1)^2} \bigg)$.
\end{theorem}

\begin{theorem} [Asymptotic behavior of the Gini coefficient for $k$ large] \label{lim_k_large}
Let $X \sim NB(k,p)$. Then, for $k \to \infty$,
\begin{align} \label{as_formula_k_large}
    \sqrt{k} \hspace{0.1cm} G(NB(k,p)) \longrightarrow \frac{\sqrt{1 + \lambda}}{\sqrt{\pi}}.
\end{align}
\end{theorem}

\begin{figure}[h!]
    \centering
    \begin{subfigure}[b]{0.3\textwidth}
        \centering            
        \includegraphics[width=\textwidth]{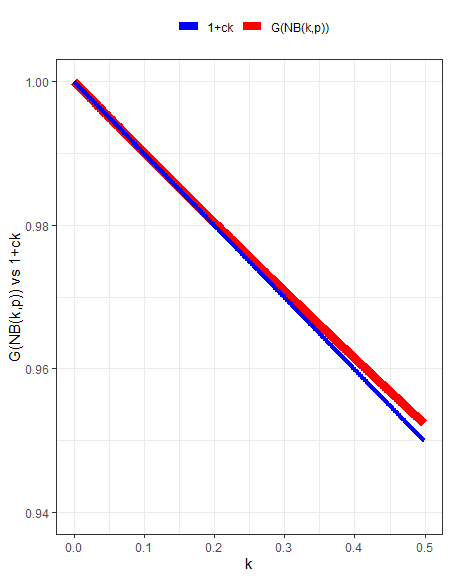}
            \caption[]%
            {{\small $p = 0.9$}}    
            %\label{fig:gini for k small zoom in close 0.9}
    \end{subfigure}
    \hfill
    \begin{subfigure}[b]{0.3\textwidth}   
        \centering 
        \includegraphics[width=\textwidth]{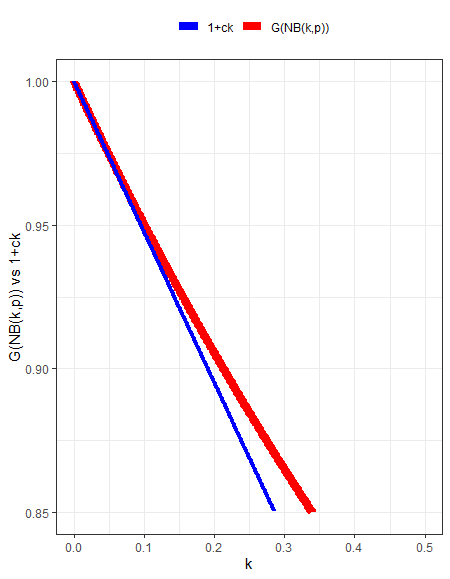}
            \caption[]%
            {{\small $p = 0.5$}}    
            %\label{fig:gini for k small zoom in close 0.5}
    \end{subfigure}
    \hfill
    \begin{subfigure}[b]{0.3\textwidth}   
        \centering 
        \includegraphics[width=\textwidth]{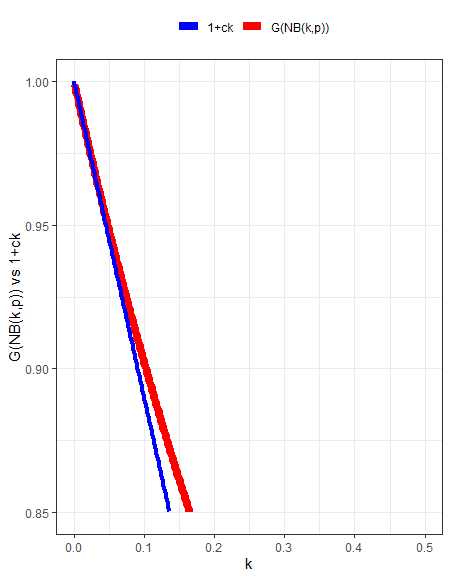}
            \caption[]%
            {{\small $p = 0.1$}}    
            %\label{fig:gini for k small zoom in close 0.1}
    \end{subfigure}
    \caption{Behavior of the Gini coefficient for $k \to 0$ vs.\ $y=1+ck$} \label{fig:gini for k small zoom in close}
\end{figure} 

\begin{figure}[h!]
    \begin{subfigure}[b]{0.3\textwidth}   
        \centering 
        \includegraphics[width=\textwidth]{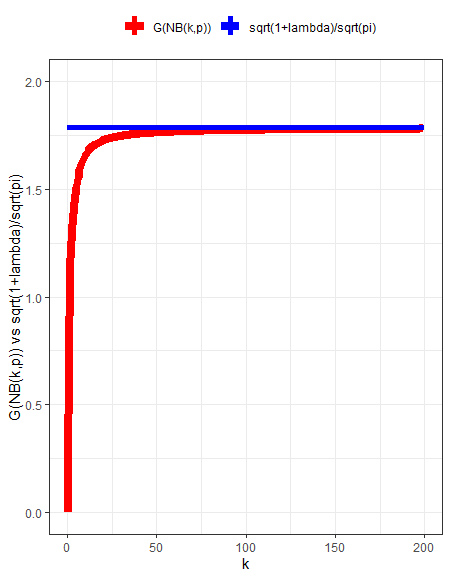}
            \caption[]%
            {{\small $p = 0.9$}}    
            %\label{fig:gini blowup 0.9}
    \end{subfigure}
    \hfill
    \begin{subfigure}[b]{0.325\textwidth}   
        \centering 
        \includegraphics[width=\textwidth]{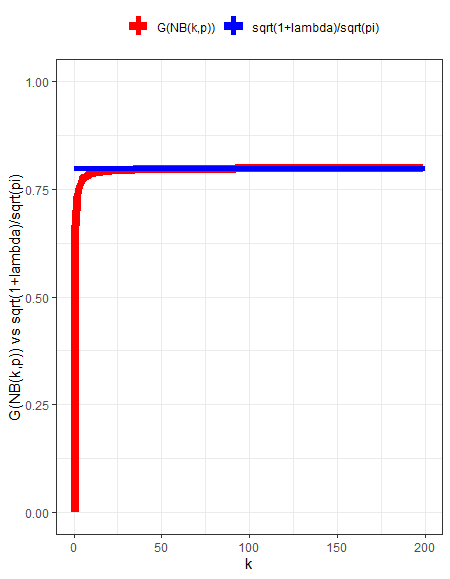}
            \caption[]%
            {{\small $p = 0.5$}}    
            %\label{fig:gini blowup 0.5}
    \end{subfigure}
    \hfill
    \begin{subfigure}[b]{0.325\textwidth}   
        \centering 
        \includegraphics[width=\textwidth]{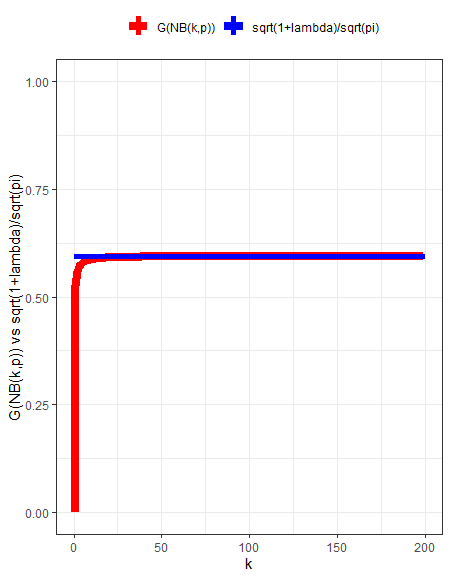}
            \caption[]%
            {{\small $p = 0.1$}}    
            %\label{fig:gini blowup 0.1}
    \end{subfigure}
    \caption{Behavior of the Gini coefficient times $\sqrt{k}$ for $k \to \infty$ vs.\ $y=\frac{\sqrt{1 + \lambda}}{\sqrt{\pi}}$} \label{fig:gini_blowup}
\end{figure}
Proofs of these theorems are given in the Appendix. We validate the asymptotic estimates in the theorems by plotting the Gini coefficient and its limiting functions against each other. See Figure \ref{fig:gini for k small zoom in close} for $k \to 0$ and Figure \ref{fig:gini_blowup} for $k \to \infty$.

\subsection{The Gini coefficient for COVID-19' secondary infections} \label{gini_data_example}

In this section we investigate how our formulas perform in a real-world example. Heterogeneity in the number of secondary cases receives significant attention during the currently ongoing pandemic of COVID-19. \cite{Gini_Indonesia} aims to quantify overdispersion by estimating $k$. We use the data provided there to first compute the Gini coefficient explicitly (by (\ref{Gini_classic})) and then calculate it using formulas we derived in this paper. For the latter, we substituted the parameters $p$ and $k$ obtained by the authors. The table below presents all results. 
\begin{table}[h!]
\begin{center}
\begin{tabular}{ |c|c|c| } 
 \hline
 \textbf{Location} & \textbf{Jakarta - Depok} & \textbf{Batan} \\ 
 \hline
 $R_0$ & 6,79 & 2,47 \\ 
 \hline
 $k$ & 0,06 & 0,2 \\ 
 \hline
 $p$ & 0,008 & 0,06 \\
 \hline
 $G(X)$ computed explicitly from data & 0,9213411 & 0,83191721 \\
 \hline
 $G(X)$ computed with (\ref{formula_one})& 0,9269144 & 0,8151066 \\
 \hline
 $G(X)$ computed with (\ref{as_formula_k_small}) & 0.9193061 & 0.7623808\\
 \hline
\end{tabular}
\caption{\label{tab1}Secondary cases data summary for Jakarta-Depok and Batan.}
\end{center}
\end{table}
These results show that our formulas come very close to the actual value of the Gini coefficient. It is also apparent that they reflect the dynamic of epidemic - lower values of $k$ assigned to higher overdispersion correspond to higher value of the Gini coefficient.

\section{Conclusion} \label{conclusion}

The models based on the negative binomial distribution of secondary infections have been used to estimate overdispersion in the spread of infectious diseases for a long time. However, this approach necessitates parametric models suitable only for specific situations. We argue that it is preferable to use the Gini coefficient to measure overdispersion. It correctly captures the dynamic of epidemic and is arguably more informative than popular so far parameter $k$ as it is solely known as the measure of variability and it is not bound to any specific probability distribution.

Thanks to the new representations that we presented in this paper, we can express the Gini coefficient as a simple function of the parameters of the distribution for multiple probability distributions, as long as they have finite expected value, making the Gini index a 
versatile statistic not depending on parametric assumptions. In particular, it can be useful in the case of heavy-tail distributions where the variance might be infinite. In \cite{Taleb} the authors present methods that allow estimating the Gini coefficient for heavy-tailed data with infinite variance.

The applications of the newly developed representations do not need to be limited to epidemiology. The Gini coefficient was advocated as an indicator in many other areas and we hope our results could help widen its usage. For reference see for example \cite{heart_rate}, where the Gini coefficient is suggested as a measure of heart rate variability to assess the level of mental stress, or \cite{White}, where the Gini coefficient is applied to explain inequalities in resource use around the globe.

\appendix
\section{Asymptotics Gini coefficient of negative binomial}

\subsection{Proof of Theorem \ref{lim_k_small}}

We begin by proving two auxiliary propositions.

\begin{proposition}[Asymptotic behavior of $\mathbf{P}(X_k=j)$] \label{prop_asymptotic_small_1} 
When $k \to 0$,
    \begin{align} \label{probab_X=0}
        \mathbf{P}(X_k=0) =  1 - k \log \big(\frac{\lambda + 1}{\lambda} \big) + o(k),
    \end{align}
    \begin{align} \label{probab_X=j}
        \mathbf{P}(X_k=j) = c_j k + o(k), \hspace{1cm} j\geq 1,
    \end{align}
    where $c_j=\frac{1}{j(\lambda+1)^j}.$
It is also true that $\mathbf{E}[X_k]$ is asymptotically equal to $k$ for $k$ small.
\end{proposition}

\begin{proof}
For $X_k \sim Poisson(\Gamma(k,\lambda))$,
\begin{align*}
    \mathbf{P}(X_k = j) & = \frac{\lambda^k}{j! \Gamma(k)}\frac{\Gamma(j+k)}{(\lambda+1)^{j+k}} = \bigg(\frac{\lambda}{\lambda+1} \bigg)^k \frac{1}{j!(\lambda+1)^j} \frac{\Gamma(j+k)}{\Gamma(k)}.
\end{align*}
Hence, for $k \to 0$,
\begin{align*}
    \mathbf{P}(X_k = 0) & = \bigg( \frac{\lambda}{\lambda + 1} \bigg)^k = e^{-k \log(\frac{\lambda + 1}{\lambda}) } = 1 - k \log \big(\frac{\lambda + 1}{\lambda} \big) + o(k),
\end{align*}
where we apply $e^{-ak} = 1 - ak + o(k)$ in the second equality.
Further, applying the same asymptotic formula together with $\Gamma(k) \sim \frac{1}{k}$ and $\Gamma(j+k)=\Gamma(j)+O(k)$ (see \cite[6.1.35]{AbramowitzStegun}), we obtain, for $j\geq 1$,
\begin{align*} 
    \mathbf{P}(X_k = j) & =\bigg( 1 - k \log \big(\frac{\lambda + 1}{\lambda} \big) + o(k) \bigg) \frac{1}{j!(\lambda+1)^j}  \big(k\Gamma(j)+o(k)\big) \\
    & = k \frac{(j-1)!}{j!(\lambda+1)^j} + o(k),
\end{align*}
which yields (\ref{probab_X=j}).
\end{proof}

\begin{proposition}[Asymptotic behavior of $\mathbf{P}(X_k^*\geq j)$] \label{prop_asymptotic_small_2}
For $k \to 0$ and $j \geq 1$,
\begin{align} \label{probab_X*>=j}
   \mathbf{P}(X_k^* \geq j) & = d^*_j + o(1),
\end{align}
with $d^*_j = \big( \frac{1}{\lambda+1} \big)^{j} - j\lambda \sum_{n=j}^{\infty} \frac{\big( \frac{1}{\lambda+1} \big)^{n+1}}{n+1} $.
\end{proposition}

\begin{proof}
We first compute for $l\geq1$ and $k \to 0$,
\begin{align*}
    \mathbf{P}(X^*_k=l) & = \frac{\mathbf{P}(X_k>l)}{\mathbf{E}X_k} = \frac{\lambda}{k} \mathbf{P}(X_k\geq l+1) = \frac{\lambda}{k}\left(1- \sum_{m=0}^{l} \mathbf{P}(X_k=m)\right) \\
    = & \frac{\lambda}{k}\bigg(1 - \sum_{m=0}^{l} [c_m k + o(k)] \bigg) = \frac{\lambda}{k} \bigg(k \log \big(\frac{\lambda + 1}{\lambda} \big) + o(k) - \sum_{m=1}^{l} c_m k + o(k) \bigg) \\
    = & \lambda \sum_{m=l+1}^{\infty}c_m + o(1) = d_l + o(1), \nonumber
\end{align*}
where 
    \[d_l = \lambda \sum_{m=l+1}^{\infty} \frac{1}{m(\lambda+1)^m} = \frac{\lambda}{(\lambda+1)^{l+1}} \sum_{n=0}^{\infty} \frac{\big( \frac{1}{\lambda+1} \big)^n}{n+l+1},
    \]
and in the fourth equality, we have substituted (\ref{probab_X=j}). Hence, in a similar manner,         \[
    \mathbf{P}(X^*_k \geq j) = 1 - \mathbf{P}(X^*_k < j)=1-\sum_{l=0}^{j-1} d_l + o(1),
    \]
and we  compute
\begin{align*}
    \sum_{l=0}^{j-1} d_l & = \sum_{l=0}^{j-1} \frac{\lambda}{(\lambda+1)^{l+1}} \sum_{n=0}^{\infty} \frac{\big( \frac{1}{\lambda+1} \big)^n}{n+l+1} = \lambda \sum_{l=0}^{j-1} \sum_{n=0}^{\infty} \frac{\big( \frac{1}{\lambda+1} \big)^{n+l+1}}{n+l+1} \\
    & = \lambda \sum_{l=0}^{j-1} \sum_{n=l}^{\infty} \frac{\big( \frac{1}{\lambda+1} \big)^{n+1}}{n+1} = j\lambda \sum_{n=j}^{\infty} \frac{\big( \frac{1}{\lambda+1} \big)^{n+1}}{n+1} + \lambda \sum_{n=0}^{j-1} \frac{\big( \frac{1}{\lambda+1} \big)^{n+1}}{n+1} (n+1) \nonumber \\
    & = j\lambda \sum_{n=j}^{\infty} \frac{\big( \frac{1}{\lambda+1} \big)^{n+1}}{n+1} + \lambda \cdot \frac{\lambda+1-\big( \frac{1}{\lambda+1} \big)^{j-1}}{\lambda(\lambda+1)}\\
    & = 1 - \bigg( \frac{1}{\lambda+1} \bigg)^{j} + j\lambda \sum_{n=j}^{\infty} \frac{\big( \frac{1}{\lambda+1} \big)^{n+1}}{n+1},
\end{align*}
from which the statement follows. 
\end{proof}
Now we can proceed with the proof of the main theorem:

\noindent
\begin{proof}[Proof of Theorem \ref{lim_k_small}]
Since $\mathbf{P}(X^*_k\geq 0)=1$ we write, for all $N$,
\begin{align*}
    G(X_k) & = \mathbf{P}(X_k=0) + \sum_{j=1}^{\infty} \mathbf{P}(X_k=j)\mathbf{P}(X^*_k\geq j) \\
    & = 1 - k \log \big(\frac{\lambda + 1}{\lambda} \big) + o(k) + \sum_{j=1}^{N} (c_jk + o(k))(d^*_j+o(1)) \\
    & + \sum_{j=N+1}^{\infty} \mathbf{P}(X_k=j)\mathbf{P}(X^*_k\geq j),
\end{align*}
where we have substituted (\ref{probab_X=0}) and (\ref{probab_X=j}) from Proposition \ref{prop_asymptotic_small_1} and (\ref{probab_X*>=j}) from Proposition \ref{prop_asymptotic_small_2}. Thus, for all $N$
\begin{align*}
    \limsup_{k \to 0} \frac{1-G(X_k)}{k} \leq \log \big(\frac{\lambda + 1}{\lambda} \big) + o(1) - \sum_{j=1}^{N} c_j d^*_j.
\end{align*}
and therefore, taking $N \to \infty$
\begin{align} \label{limsup}
    \limsup_{k \to 0} \frac{1-G(X_k)}{k} \leq \log \big(\frac{\lambda + 1}{\lambda} \big) + o(1) - \sum_{j=1}^{\infty} c_j d^*_j.
\end{align}
On the other hand, for all $N$ we have
\begin{align*}
    \liminf_{k \to 0} \frac{1-G(X_k)}{k} & \geq \log \big(\frac{\lambda + 1}{\lambda} \big) + o(1) - \sum_{j=1}^{N} c_j d^*_j - \frac{1}{k} \sum_{N+1}^{\infty} \mathbf{P}(X_k=j)\\
    & \geq \log \big(\frac{\lambda + 1}{\lambda} \big) + o(1) - \sum_{j=1}^{N} c_j d^*_j - \frac{1}{k} \frac{\mathbf{E}[X_k]}{N+1}\\
    & = \log \big(\frac{\lambda + 1}{\lambda} \big) + o(1) - \sum_{j=1}^{N} c_j d^*_j - \frac{1}{\lambda(N+1)},
\end{align*}
where we have applied $\mathbf{P}(X_k=j) \leq 1$ and Markov inequality. Hence, letting $N \to \infty$,
\begin{align} \label{liminf}
    \liminf_{k \to 0} \frac{1-G(X_k)}{k} \geq \log \big(\frac{\lambda + 1}{\lambda} \big) + o(1) - \sum_{j=1}^{\infty} c_j d^*_j.
\end{align}
Combining (\ref{liminf}) and (\ref{limsup}) we can write
\begin{align} \label{k_small_gini_eq}
    G(X_k) = 1 + k \bigg(\sum_{j=1}^{\infty}c_j d^*_j - \log \big(\frac{\lambda + 1}{\lambda} \big) \bigg) + o(k).
\end{align}
Now we compute the term $\sum_{j=1}^{\infty}c_j d^*_j$. We have
\begin{align} \label{k_small_const1}
    \sum_{j=1}^{\infty}c_j d^*_j & = \sum_{j=1}^{\infty} \frac{1}{j(\lambda+1)^j} \bigg( \bigg( \frac{1}{\lambda+1} \bigg)^{j} - j\lambda \sum_{n=j}^{\infty} \frac{\big( \frac{1}{\lambda+1} \big)^{n+1}}{n+1} \bigg) \nonumber\\
    & = \sum_{j=1}^{\infty} \frac{1}{j(\lambda+1)^{2j}} - \lambda \sum_{j=1}^{\infty} \frac{1}{(\lambda+1)^j} \sum_{n=j}^{\infty} \frac{\big( \frac{1}{\lambda+1} \big)^{n+1}}{n+1}.
\end{align}
The first summand in (\ref{k_small_const1}) is easily computable and equals $-\log\big( \frac{\lambda^2+2\lambda}{(\lambda+1)^2} \big)$. The second summand can be computed by interchanging the order of summation, as
\begin{align*}
    \lambda & \sum_{j=1}^{\infty} \frac{1}{(\lambda+1)^j} \sum_{n=j}^{\infty} \frac{\big( \frac{1}{\lambda+1} \big)^{n+1}}{n+1} = \lambda \sum_{n=1}^{\infty} \frac{\big( \frac{1}{\lambda+1} \big)^{n+1}}{n+1} \sum_{j=1}^n \frac{1}{(\lambda+1)^j}\\
    & = \sum_{n=1}^{\infty} \frac{\big( \frac{1}{\lambda+1} \big)^{n+1}}{n+1} \bigg(1 - \big( \frac{1}{\lambda+1} \big)^n \bigg) = \sum_{n=1}^{\infty} \frac{\big( \frac{1}{\lambda+1} \big)^{n+1}}{n+1} - \sum_{n=1}^{\infty} \frac{\big( \frac{1}{\lambda+1} \big)^{2n+1}}{n+1} \\
    & = -\log\big(\frac{\lambda}{\lambda+1}\big) - \frac{1}{\lambda+1} + \big(\lambda+1\big) \log\bigg(\frac{\lambda^2+2\lambda}{(\lambda+1)^2}\bigg) + \frac{1}{\lambda+1} \\
    & = \big(\lambda+1\big) \log\bigg(\frac{\lambda^2+2\lambda}{(\lambda+1)^2}\bigg)-\log\big(\frac{\lambda}{\lambda+1}\big).
\end{align*}
Hence, (\ref{k_small_gini_eq}) becomes
\begin{align*} 
    G(X_k) = 1 + k \bigg(2\log \bigg(\frac{\lambda}{\lambda+1}\bigg) - (\lambda+2)\log\bigg(\frac{\lambda^2+2\lambda}{(\lambda+1)^2} \bigg) \bigg) + o(k),
\end{align*}
what can be also written as
\begin{align*}
    G(X_k) = 1 + k \big[2\log(p)+\frac{2-p}{1-p}\log\big(2p(1-p)\big) \big] + o(k),
\end{align*}
when we take $p=\frac{\lambda}{\lambda+1}$.
\end{proof}

\subsection{Proof of Theorem \ref{lim_k_large}}
We first prove the following proposition:
\begin{proposition} \label{prop_1_klarge}
Denote $X_k = N(\Gamma_k)$. For $k \to \infty$,
\begin{align*} 
  \frac{X_k - \frac{k}{\lambda} }{\sigma \sqrt{k}} \stackrel{d}{\longrightarrow} N(0,1) \hspace{0.5cm} \text{as} \hspace{0.3cm} k \to \infty,
\end{align*}
where $\sigma = \frac{\sqrt{1 + \lambda}}{\lambda}$.
\end{proposition}
\begin{proof} We can write
\begin{align*}
    \frac{N(\Gamma_k) - \frac{k}{\lambda} }{\sqrt{k}} = \frac{N(\Gamma_k) - \Gamma_k}{\sqrt{\Gamma_k}} \cdot \sqrt{\frac{\Gamma_k}{k}} + \frac{\Gamma_k - \frac{k}{\lambda}}{\sqrt{k}}.
\end{align*}
Denote $\Tilde{\Gamma}_k = \frac{\lambda\Gamma_k-k}{\sqrt{k}}$. From the central limit theorem, it follows that
\begin{align*}
    \mathbf{P}(\Tilde{\Gamma}_k \geq x) \to \Bar{\Phi}(x) \hspace{0.5cm} \text{as} \hspace{0.3cm} k \to \infty.
\end{align*}
Hence,
\begin{align*}
    \frac{\Gamma_k - \frac{k}{\lambda}}{\sqrt{k}} = \frac{\lambda\Gamma_k-k}{\lambda \sqrt{k}} \stackrel{d}{\longrightarrow} W_{\frac{1}{\lambda^2}} \hspace{0.5cm} \text{as} \hspace{0.3cm} k \to \infty.
\end{align*}
Furthermore, by the strong law of large numbers
\begin{align*}
    \frac{\Gamma_k}{k} \to \frac{1}{\lambda} \hspace{0.5cm} \text{as} \hspace{0.3cm} k \to \infty,
\end{align*}
and thus $\sqrt{\frac{\Gamma_k}{k}} \to \frac{1}{\sqrt{\lambda}}.$ Finally, $\Gamma_k \to \infty$ \textit{a.s.} as $k \to \infty$ and again, by the central limit theorem we know that for a Poisson process $N(k)$ it is true that
\begin{align*}
    \frac{N(k)-k}{\sqrt{k}} \stackrel{d}{\longrightarrow} W_1 \hspace{0.5cm} \text{as} \hspace{0.3cm} k \to \infty.
\end{align*}
We conclude that also $\frac{N(\Gamma_k) - \Gamma_k}{\sqrt{\Gamma_k}} \stackrel{d}{\longrightarrow} W_1$ as $k \to \infty$. The claim follows.
\end{proof}

\begin{proof}[Proof of Theorem \ref{lim_k_large}]
Recall $X_k = N(\Gamma_k)$ and set $\Tilde{X}_k = \frac{X_k - \frac{k}{\lambda}}{\sigma \sqrt{k}}$. From (\ref{Gini_classic}) we obtain
\begin{align*}
    \sqrt{k} G(X_k) = \frac{\mathbf{E}\big[{|X_k - X'_k|}\big]}{2\mathbf{E}[X]} = \frac{\lambda}{2\sqrt{k}} \mathbf{E}\big[{|X_k - X'_k|}\big],
\end{align*}
where $X'_k$ is an independent copy of $X_k$. Then, for fixed $M$,
\begin{align*}
    \sqrt{k} G(X_k) & = \frac{\lambda}{2\sqrt{k}} \mathbf{E}\big[{|X_k - X'_k|} \cdot \mathbf{1}_{\{|X_k - X'_k|<M \sqrt{k}\}} \big] \\
    & + \frac{\lambda}{2\sqrt{k}} \mathbf{E}\big[{|X_k - X'_k|} \cdot \mathbf{1}_{\{|X_k - X'_k|>M \sqrt{k}\}}\big] = I + II.
\end{align*}
We study parts $I$ and $II$ separately, starting with $I$:
\begin{align*}
    I & = \frac{\lambda \sigma \sqrt{k}}{2\sqrt{k}} \mathbf{E}\big[{|\Tilde{X}_k - \Tilde{X'}_k|} \cdot \mathbf{1}_{\{|\Tilde{X}_k - \Tilde{X'}_k|< \frac{M\sqrt{k}}{\sigma \sqrt{k}}\}} \big]\\
    & =  \frac{\sigma \lambda}{2} \mathbf{E}\big[{|\Tilde{X}_k - \Tilde{X'}_k|} \cdot \mathbf{1}_{\{|\Tilde{X}_k - \Tilde{X'}_k|< \frac{M}{\sigma}\}} \big]
\end{align*}
Letting first $k$ and then $M$ tend to infinity, by Proposition \ref{prop_1_klarge}, we arrive at
\begin{align*}
    I \to \frac{\sigma \lambda}{2} \mathbf{E}\big[{|\Tilde{X} - \Tilde{X'}|} \cdot \mathbf{1}_{\{|\Tilde{X} - \Tilde{X'}|< \infty \}}\big],
\end{align*}
where $\Tilde{X}$ and $\Tilde{X'}$ are independent standard normal random variables. Write $Y = \Tilde{X} - \Tilde{X'}$. Then $Y \sim N(0,2)$ and $\mathbb{E}[|Y|] = \frac{2}{\sqrt{\pi}}$. Thus,
\begin{align} \label{partI}
    I \to \frac{\sigma \lambda}{\sqrt{\pi}} = \frac{\sqrt{1+\lambda}}{\sqrt{\pi}} . 
\end{align}
For term $II$,
\begin{align} \label{partII}
    II & = \frac{\lambda}{2\sqrt{k}} \mathbf{E}\big[{|X_k - X'_k|} \cdot \mathbf{1}_{\{|X_k - X'_k|>M \sqrt{k}\}}\big] \leq \frac{\lambda}{2\sqrt{k}} \frac{\mathbf{E}\big[{(X_k - X'_k)^2}\big]}{M \sqrt{k}} \nonumber \\
    & = \frac{\lambda}{2\sqrt{k}} \frac{2\Var(X_k)}{M\sqrt{k}} = \frac{\lambda}{2\sqrt{k}} \frac{2}{M\sqrt{k}} \frac{k(1+\lambda)}{\lambda^2} = \frac{1+\lambda}{M\lambda} \stackrel{M \to \infty}{\longrightarrow} 0.
\end{align}
Combining (\ref{partI}) and (\ref{partII}) we obtain the desired claim (\ref{as_formula_k_large}) from Theorem \ref{lim_k_large}.
\end{proof}

\begin{wrapfigure}{r}{5cm}
\label{wrap-fig:1}
\includegraphics[width=0.225\textwidth]{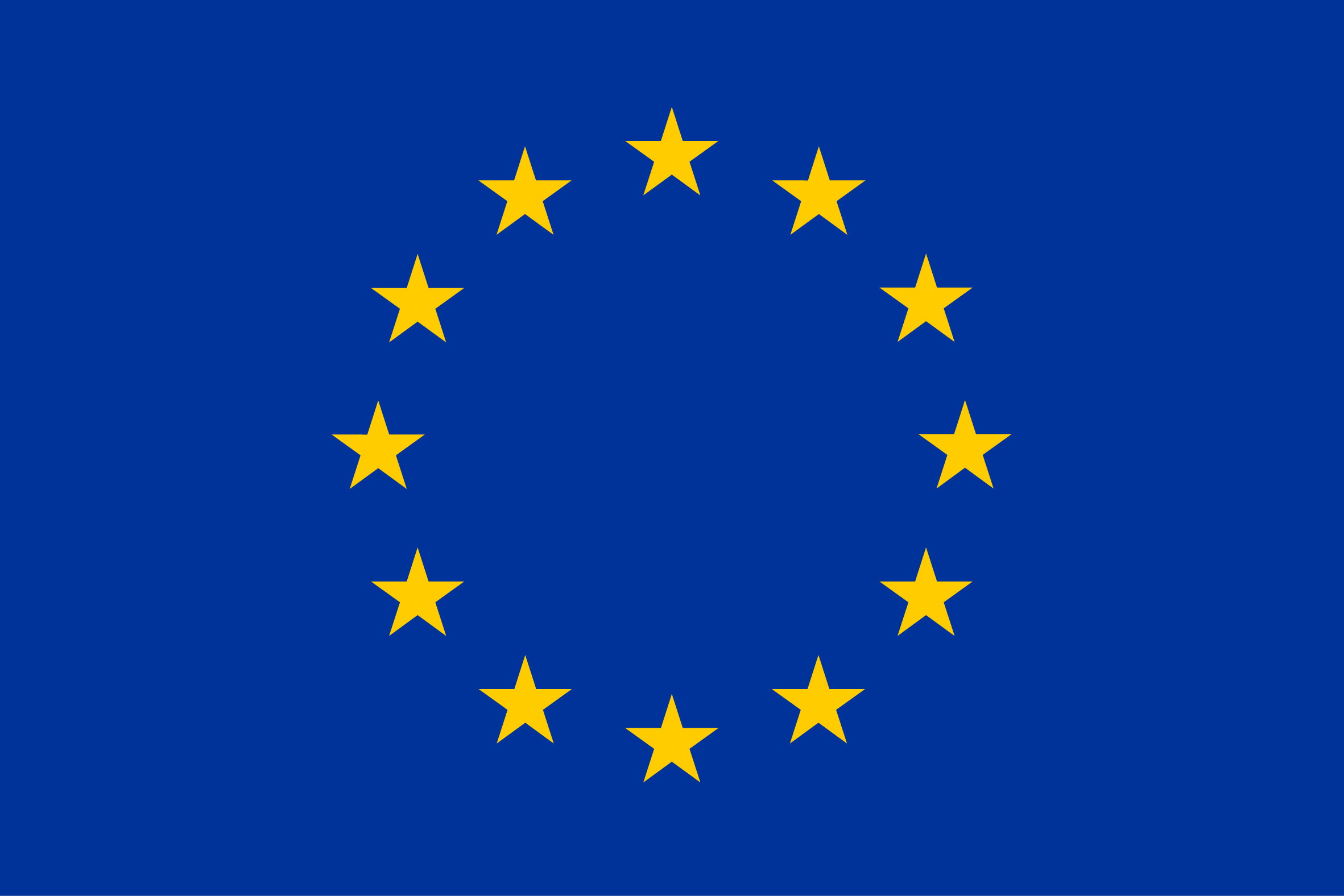}
\end{wrapfigure} 
\paragraph{\bf Acknowledgement.} The work of MM is supported by the European Union’s Horizon 2020 research and innovation programme under the Marie Skłodowska-Curie grant agreement no. 945045, and by the NWO Gravitation project NETWORKS under grant no. 024.002.003.\\
The work of RvdH is supported in parts by the NWO through the Gravitation {\sc Networks} grant 024.002.003.

\end{document}